\documentclass[12pt]{amsart}

\usepackage{amssymb}

\newtheorem{theorem}{Theorem}[section]

\newtheorem{proposition}[theorem]{Proposition}

\theoremstyle{definition}

\newtheorem{definition}[theorem]{Definition}

\newtheorem{remark}[theorem]{Remark}

\newtheorem{example}[theorem]{Example}

\theoremstyle{parrafo}

\newcommand{\R}{\mathbb{R}}

\begin{document}

\title[]{Besicovitch and doubling type properties in metric spaces}

\author{J. M. Aldaz}
\address{Instituto de Ciencias Matem\'aticas (CSIC-UAM-UC3M-UCM) and Departamento de 
Mate\-m\'aticas,
Universidad  Aut\'onoma de Madrid, Cantoblanco 28049, Madrid, Spain.}
\email{jesus.munarriz@uam.es}
\email{jesus.munarriz@icmat.es}


\thanks{The author was partially supported by Grant PID2019-106870GB-I00 of the
MICINN of Spain, and also  by ICMAT Severo Ochoa project 
CEX2019-000904-S (MICINN)}







\begin{abstract} We explore the relationship in metric spaces between
different properties related to the Besicovitch covering
theorem, and also consider weak versions of doubling, in connection to the non-uniqueness of centers and radii in arbitrary metric spaces.

\end{abstract}


\maketitle


\markboth{J. M. Aldaz}{Besicovitch and doubling type  properties}

\section {Introduction} 

It is well known that in metric spaces neither centers not radii of balls are in general unique. For instance, every point in a ball of an ultrametric space is a center of the said
ball; less exotic examples can be found simply as subsets of $\R^d$ with the inherited distance. This issue is often little emphasized: consider for example the case of doubling measures, which are sometimes defined by requiring the existence of a constant $C \ge 1$ such that for all balls $B$, the measure $\mu$ satisfies  $\mu (2 B ) \le C \mu (B)$. Of course, $2 B$ is not intrinsically defined in terms of $B$ and may refer to several different balls, unless its center and its radius are unique.  

In this paper we consider Besicovitch and doubling type  properties in metric spaces, devoting more attention than customary to the uniqueness issue for centers and radii.

While the  Besicovitch covering theorem has several important consequences, 
the spaces for which it holds have often been regarded as being rather special, cf. for example \cite[pp. 7-8]{He}.  
It is thus natural to explore different properties of Besicovitch type,  see for instance \cite{LeRi}, \cite{LeRi2}. 
Recall that in the Besicovitch covering theorem one is given a set $A$ covered by a collection $\mathcal{C}$ of
centered balls, with uniformly bounded radii. The 
strong Besicovitch covering property takes the conclusion of the Besicovitch covering  theorem, regarding the existence of a uniformly bounded 
number of disjoint subcollections whose union still covers $A$, and makes it into a definition. The weaker notion of Besicovitch covering property, asks for the existence of a covering subcollection having uniformly 
bounded overlaps, so the intersections at any given point cannot
exceed a certain constant. The  Besicovitch intersection property is even weaker: given any collection of balls such that
no ball contains a suitably chosen center of another, it is required that the  overlap be uniformly bounded (there is no mention of any
set to be covered). The precise statements for these notions appear in Definition \ref{BCP}. 
Consideration of weak properties of Besicovitch type is interesting since more spaces  satisfy them, they are easier to 
prove, and for certain purposes  they
may be sufficient. For instance, the Besicovitch intersection property in metric measure spaces implies  the uniform weak type (1,1) of the centered maximal operator, cf. \cite{Al} (actually,  both conditions are equivalent).

It is shown in 
\cite[Example
3.4]{LeRi2} that the Besicovitch intersection property does not imply the Besicovitch covering property.
Here we complement this result by proving that the Besicovitch covering property does not imply the strong Besicovitch covering property, cf. Theorem \ref{notSBCP} below.

Covering arguments in connection with the boundedness properties of maximal operators often blow up radii or reduce them.  If one is not concerned with optimal constants and the measure is doubling, it does not really matter which center and radius is chosen for a ball $B$, as all balls that can be written as $2B$ will have comparable measure. However, the situation may change if one is trying to obtain sharper bounds.
In particular, with regard to the centered maximal operator, improving bounds may require reducing radii, a technique initiated in
\cite[Lemma]{StSt}.  Now, the result of considering smaller radii will
depend on the choice of center and radii. In fact, by reducing radii a ball $B$ can give rise to several disjoint smaller balls, concentric with $B$. To carry out certain covering arguments, it can be necessary to consider all the different centers and radii a ball may have.

The definition of doubling requires that the bound $\mu (B (x, 2r )) \le C \mu (B)$ hold
for all choices of $x$ and $r$ such that $B = B(x,r)$. We consider weaker variants where the existence of bounds is only assumed for some such choice. 
A modification of the standard covering argument shows that if for every ball $B$ it is always possible
to find a center $x$ and a radius $r$ such that $B = B(x,r)$ and $\mu (B (x, (1 + \sqrt{2}) r )) \le C \mu (B)$, then the uncentered maximal operator is of weak type $(1,1)$ with constant $C^2$, cf. Theorem \ref{HLMF}.

  \section{Properties  related to  the Besicovitch covering theorem} 
  
    By a disjoint collection of sets, we mean that all the sets in such a  collection are disjoint. Unless otherwise stated we will assume that metric spaces, collections of balls, etc, are nonempty, in order to avoid trivialities.

  The  strong Besicovitch covering property,   as far as I know, 
appears for the first time in \cite[Theorem 4.2 ii)]{Ri}. The  notion of Besicovitch covering property  is taken from \cite{LeRi}, but the
  condition had been utilized before.  The 
  definition of   Besicovitch intersection  property  comes essentially from  \cite{LeRi} (where it is called 
  the weak Besicovitch covering property).

  \begin{definition} \label{amp}  A metric space $X$ is   {\em ultrametric} 
if the standard triangle inequality
	is replaced by the stronger condition $d(x,y) \le \max\{d(x,w), d(w,y)\}$.
\end{definition}

As standard examples of ultrametric spaces, coming from Number Theory, we mention the $p$-adic metric spaces.

\vskip .2 cm

We will use $B^{\operatorname{o}}(x,r) := \{y\in X: d(x,y) < r\}$ to denote metrically open balls, 
and 
$B^{\operatorname{cl}}(x,r) := \{y\in X: d(x,y) \le r\}$ to refer to metrically closed balls;
open and closed will
always be understood  in the metric (not the topological) sense. 
If we do not want to specify whether balls are open or  closed,
we write $B(x,r)$. But when we utilize $B(x,r)$,  all balls are taken to be of the same kind, i.e., all open or all closed.

\begin{definition}  \label{cover} 
	Let  $(X, d)$ be a   metric space, and let  $\mathcal{C}$ be a collection of balls. We say that   $\mathcal{C}$ is {\em uniformly bounded}  if there exists an $R > 0$ such that given any ball
	$B \in \mathcal{C}$, there is a choice of center $x$ and of radius $r$ with $B = B(x,r)$ and $r \le R$.
		
	Additionally,
	 we say   $\mathcal{C}$ is a {\em centered cover} of $A$ if  for every $x\in A$,  there exist a ball
	$B \in \mathcal{C}$ such that $x$ is a center of $B$.
\end{definition}

\begin{example} Let $X = (0,1)$. Then $\mathcal{C} = \{B(x,1): x \in X\} = \{(0,1)\}$ 
is a uniformly bounded centered cover of $X$, which contains only one element (with uncountably many names $B(x,1)$). Note that given a collection of balls,  we can always remove the non-uniqueness issue by using the axiom of choice, selecting exactly one center and one radius for each ball, and then continue arguing as if centers and radii were unique, by sticking to that choice. However, in that case a collection such as $\mathcal{C}$ would be a centered cover not of the full interval $(0,1)$ but only of one point, its preassigned center. Note also that the definition of uniformly bounded cover only requires the existence of some choice of radii that satisfies the condition. For instance, we could have written $\mathcal{C} = \{B(x,1/x): x \in X\} = \{(0,1)\}$ and $\mathcal{C}$ 
would still be a uniformly bounded centered cover of $X$.
\end{example}

\begin{definition}  \label{BCP} 
A   metric space $(X, d)$ 
	 has the {\em strong Besicovitch covering property}  if there exists a constant
$L \ge 1$ such that for every  $A \subset X$ and every   uniformly bounded centered cover
$\mathcal{C}$ of $A$,
for some $m  \le L$
	there are disjoint subcollections  $\mathcal{C}_1 , \dots  , \mathcal{C}_m$  of $\mathcal{C}$ 
	such that 
	$A \subset \cup_{i = 1}^m \cup \mathcal{C}_i$.

We say that $(X, d)$ 
	 has the {\em Besicovitch covering property}  if there exists a constant
$L \ge 1$ such that for every  $A \subset X$, and every  uniformly bounded  centered cover
$\mathcal{C}$ of $A$,
 there is a subcollection $\mathcal{C}^\prime \subset \mathcal{C}$  satisfying
$$
\mathbf{1}_A\le	\sum_{B \in \mathcal{C}^\prime} \mathbf{1}_{B} 
\le L.
$$ 

A  collection $\mathcal{C}$ of balls in a metric space $(X, d)$ is a {\em Besicovitch family} if for each  $B \in \mathcal{C}$ it is possible to choose a center and a radius,
 such that whenever $B_1, B_2 \in \mathcal{C}$ are distinct balls, neither ball contains the selected center of the other.
 
The space $(X, d)$ has the {\em   Besicovitch
		intersection property}  with constant $L$,  if there exists an integer $L\ge 1$ 
		such that 
	 for every  Besicovitch family $\mathcal{C}$,  we have 
	$$
	\sum_{B \in \mathcal{C}} \mathbf{1}_{B} 
	\le L.
	$$
	\end{definition}

\begin{example} For a collection of balls $\mathcal{C}$ to be a Besicovitch family it is enough that some suitable assignment of centers and radii exists. 
For instance, let $X = [0,1]$
with the distance inherited from the line, let $B_1 = [0,3/4]$ and
let $B_2 = [1/4, 1]$. Then $\{B_1, B_2\}$ is  a
Besicovitch family,  since  we can write $B_1 = B^{\operatorname{cl}}(0,3/4)$ and $B_2 = B^{\operatorname{cl}}(1,3/4)$. Not all choices work: if we use  $B_1 = B^{\operatorname{cl}}(1/4, 1/2)$ and $B_2 = B^{\operatorname{cl}}(3/4,1/2)$, then the condition on the centers is not fulfilled. 
\end{example}

\begin{remark} Different metrics on the same set can define the same class of balls, and still with one metric the space satisfies all the Besicovitch type properties considered above, and with the other metric, all fail. This happens to the open upper half plane, with the Euclidean and hyperbolic metrics (cf. \cite[p. 237]{McC}): with the Euclidean metric the strong Besicovitch covering holds by Besicovitch Theorem, and hence, so do the weaker variants, while the Besicovitch intersection property fails in the hyperbolic plane, by the exponential growth of the area of balls with large radii, so it also lacks the stronger properties.
\end{remark}

\begin{remark} It is obvious that if $(X, d)$ has the strong Besicovitch covering property with constant $L$, 
then it has the Besicovitch covering property with constant $L$; and  it is almost obvious that if $(X, d)$ has the 
 Besicovitch covering property with constant $L$, 
then it has the Besicovitch intersection property with the same constant. The reason why the last assertion is not
entirely  obvious
	is that  radii are not assumed to be uniformly bounded in the
	definition of the Besicovitch intersection property. But this problem is
easily avoided by passing to an intersecting   finite subcollection:  let
	 $\mathcal{C}$ be an intersecting Besicovitch family of cardinality larger than
	 $L$,  and let $y \in \cap \mathcal{C}$ . We may suppose, by  throwing away some balls if needed,
	that $\mathcal{C} = \{B(x_1, r_1), \dots ,B(x_{L + 1}, r_{L + 1})\}$. 
	Let $R:= \max\{r_1, \dots , r_{L + 1}\}$, and let 
	$A := \{x_1, \dots , x_{L + 1}\}$. Since all balls in
	$\mathcal{C}$ are needed to cover $A$, and they all intersect at $y$, we conclude
	that  $(X, d)$ does not have the Besicovitch covering property with constant $L$.

\end{remark}

Next we cite, with minor notational modifications,  Example
3.4 from
\cite{LeRi2}, and prove that
 the space presented there is ultrametric (this is not shown in \cite{LeRi2}). Theorem
\ref{counter07}  contradicts \cite[Example 2.1]{KlTo}, where it is asserted
that every separable ultrametric space satisfies the Besicovitch covering property.

\begin{theorem} \label{counter07} There exists a separable  ultrametric space without
	the Besicovitch covering property.
\end{theorem}

\begin{proof} Let $X :=  \mathbb{N}\setminus \{0\}$ with the distance
		$d(i,i):= 0$, and for $i \ne j$, 
	$d(i,j):= 1 - 1/\max\{i,j\}$. To see that $(X, d)$ is 
	an ultrametric space, select three different points $i, j, k \ge 1$ and
	note that
	$$
	d(i,j):= 1 - \frac{1}{\max\{i,j\}} \le 1 - \frac{1}{\max\{i,j, k\}}
$$
$$	= \max \left\{ 1 - \frac{1}{\max\{i,k\}}, 1 - \frac{1}{\max\{k,j\}}\right\}
=  \max \left\{ d(i,k), d(k,j)\right\}.
	$$
	For a proof of the fact that $X$ does not satisfy the
	Besicovitch covering property, we refer the reader to  \cite[Example 3.4]{LeRi}.\end{proof}

In an ultrametric space, every point in a ball  is a center, and so,
if two balls intersect, then one of them contains the other.
Thus, the following is immediate:

\begin{proposition} \label{1ball} Any  intersecting  Besicovitch  family in an ultrametric space 
contains exactly one ball.
\end{proposition}

\begin{proposition}  \label{equivBesicovitchultra} 
	Let $(X, d)$ be an ultrametric space satisfying the
	Besicovitch covering property 
	with  constant $L$. Then $(X, d)$ has both 
	the  Besicovitch covering property 
	and the strong Besicovitch covering property with constant $1$.
\end{proposition} 

\begin{proof} 	Let $(X, d)$ 
	be an ultrametric space and let it satisfy the Besicovitch Covering Property with constant
	$L$. Given   $A \subset X$, and a centered, uniformly bounded cover
	$\mathcal{C}
	$ of $A$, label all its balls using a suitable choice function. By hypothesis 
	there is a subcollection $\mathcal{C}^\prime \subset \mathcal{C}$  satisfying
	$$
	\mathbf{1}_A\le	\sum_{B(x,  r) \in \mathcal{C}^\prime} \mathbf{1}_{B(x,  r)} 
	\le L.
	$$
For each  $x \in A$,
	from the finite collection of balls in  $\mathcal{C}^\prime$ containing $x$ select
the	ball with maximal radius, say  $B(y_x, r_x)$, and disregard all the others (note that since the space is ultrametric, there is only one ball of maximal radius in $\mathcal{C}^\prime$ containing $x$).  
Let $\mathcal{C}^{\prime\prime} := \{B(y_x, r_x): x \in A\}$. Then  $\mathcal{C}^{\prime\prime}$ is disjoint, since given two different intersecting balls, one must be strictly contained in the other, contradicting the maximal radius choice for the points in the smaller ball.
\end{proof}

\begin{theorem} \label{notSBCP}  There exists a separable  metric space which satisfies the Besicovitch covering property, but for which the strong Besicovitch 	covering property fails.
\end{theorem}

\begin{proof}   We take $X := \mathbb{N}^2$ with the  metric $d$ defined next.
 For all $x \in X$, we set $d(x,x) = 0$. Suppose  that $x \ne y$. If
 $x$ and $y$  belong to the same
 horizontal or to the same vertical copy of $\mathbb{N}$, that is,
 if either $x=(x_1, y_1)$ and $y = (x_2, y_1)$, or 
 $x=(x_1, y_1)$ and $y = (x_1, y_2)$, we set $d(x,y) = 1$; in
 all other cases,  we set $d(x,y) = 2$. In order to check that  $d$ is a metric,
 the only nontrivial statement we have to verify, is the triangle inequality when
 the three points $x,y,z$ are different. But then
 $d(x,y) \le 2 \le d(x,z) + d(z,y)$.

 In this example it seems
 more natural to consider metrically closed balls, and we will do so,
 though of course there is no difference between, say,  a metrically closed ball of
 radius 1 and an open ball of radius $3/2$. Note also that when a radius of a ball is $< 2$, the center of the said ball is unique.
 
  We define a choice function $\varphi$ (which assigns a center and a radius to each ball) on the set of all closed balls  $B^{\operatorname{cl}} \subset X$ as follows: $\varphi (X) := ((0,0), 2)$; if $B = \{(u,v)\}$, we set $\varphi (B) := ((u,v), 1/2)$; and for $(u, v) \in \mathbf{N}^2$,  if $B = \{ (u, w)  : w \in \mathbf{N}\} \cup  \{ (s, v)  : s \in \mathbf{N}\} $,  we set $\varphi (B) := ((u,v), 1)$. From now on, centers and radii refer to the preceding choices.
 
 Let us say that a collection of points is in general position if any two
 of them have different first coordinates, and different second coordinates;
  thus, points in general position are at distance two from each other.

 Given  three
  unit balls $B^{\operatorname{cl}}(x,1)$,  $B^{\operatorname{cl}}(y,1)$ and $B^{\operatorname{cl}}(z,1)$ with centers in general
 position, it is  clear that 
  $B^{\operatorname{cl}}(x,1) \cap B^{\operatorname{cl}}(y,1) \cap B^{\operatorname{cl}}(z,1) = \emptyset$. So, 
  given any  index set $I$  and any collection $\{B^{\operatorname{cl}}(x_i,1): i \in I \}$ of unit balls with centers in general position, 
 we have   $\sum_{i \in I} \mathbf{1}_{B^{\operatorname{cl}}(x_i,1)} \le 2$.

 Now it is easy to check that the Besicovitch covering property
 holds: for any $A \subset X$ and any centered cover
$\mathcal{C}$
of $A$, with the radii determined by $\varphi$, 
we define 
$\mathcal{C}^\prime$ as follows: if $\mathcal{C}$ contains the ball  $B^{\operatorname{cl}}((0,0),2)$,   set $\mathcal{C}^\prime := \{B^{\operatorname{cl}}((0,0),2)\}$. 
Next, we assume that  all balls in $\mathcal{C}$ have radii bounded
by 1. If all balls have radii strictly smaller than 1, then they are all
disjoint, and we set $\mathcal{C}^\prime = \mathcal{C}$. If
$\mathcal{C}$ contains at least one ball of radius exactly 1,
we construct $\mathcal{C}^\prime$ as follows: order lexicographically the
centers of the balls of radius 1, so $(x_1, y_1) <  (x_2, y_2)$ precisely when
$x_1 < x_2$ or when $x_1 = x_2$ and $y_1 < y_2$.  Let 
$\{B^{\operatorname{cl}}(w_i,1): i \in I \subset \mathbb{N} \}$ be the collection of unit balls in 
$\mathcal{C}$,
ordered according to their centers. Choose $B^{\operatorname{cl}}(w_{i_0},1) := B^{\operatorname{cl}}(w_0,1)$, and
supposing the balls $B^{\operatorname{cl}}(w_{i_j},1)$  have been selected for $1 \le j \le k$, 
let $B^{\operatorname{cl}}(w_{i_{k + 1}},1)$  be the first ball in the list centered
at a point of $A \setminus \cup_{j = 0}^k B^{\operatorname{cl}}(w_{i_j},1)$. Continue until
all points in $A$ that are centers of balls of radius 1 in 
$\mathcal{C}$ have been covered. If the chosen family     
$\{B^{\operatorname{cl}}(w_{i_j},1): j \in J \subset \mathbb{N} \}$  does not
cover $A$, for each $y \in A \setminus \cup  \{B^{\operatorname{cl}}(w_{i_j},1): j \in J\}$
choose $B^{\operatorname{cl}}(y,1/2)$, and let 
$\mathcal{C}^\prime$ be $\{B^{\operatorname{cl}}(w_{i_j},1): j \in J\} \cup
 \{B^{\operatorname{cl}}(y,1/2): y \in A \setminus \cup  \{B^{\operatorname{cl}}(w_{i_j},1): j \in J\}\}$.
 Since the selected balls of radii $1/2$ are all disjoint, and do not intersect
 the chosen balls of radius 1, we have that 
 $$
\mathbf{1}_A\le	\sum_{B^{\operatorname{cl}} (x,  r) \in \mathcal{C}^\prime} \mathbf{1}_{B^{\operatorname{cl}}(x,  r)} 
\le 2.
$$ 
However, $(X,d)$ does not satisfy the  strong Besicovitch 
 covering property with $r = 1$, since any two unit
 balls with centers in general position intersect at two points.  We give
 more details: let $A := \{(j,j): j \in \mathbb{N}\}$ be the main diagonal
 of $X$, and let 
 $
 \mathcal{C}
= \{B^{\operatorname{cl}}((j,j),1) : j \in \mathbb{N}\}
$. Since each point in $A$ belongs to only one ball, any covering subcollection
must be $\mathcal{C}$ itself. But if $i < j$, then
 $B^{\operatorname{cl}}((i,i),1) \cap B^{\operatorname{cl}}((j,j),1) = \{(i,j), (j,i)\}$, so any disjoint subcollection
 can contain at most one ball.
   \end{proof}

 \section{Weak doubling conditions for measures} 
 
Balls are sets of points. When centers and radii are not unique, we can regard a given choice of a radius and a center as a way of naming the ball. Here we note that standard doubling can be relaxed via suitable choices of ``names''.
To clarify this statement and motivate the definition below of weak $t$-bling, consider the following trivial example: let $X := \{0,1\}$ with the usual distance and the Dirac delta measure $\delta_1$, which fails to be doubling since
$\delta_1 (B(0, 3/4)) = 0$ and $\delta_1 (B(0, 3/2)) = 1$.
However,  standard covering arguments for doubling measures can be carried out simply by choosing a better name for the ball $\{0\}$, such as for instance, $\{0\} = B(0, 1/4)$. With this selection of radius we have $\delta_1 (B(0, 3/4)) = \delta_1 (B(0, 1/4)) = 0$ and so,  the tripling condition is satisfied  in a weak sense by $\delta_1$.

Let us recall the notion of metric measure space.
 
\begin{definition}  \label{tau} Let $(X, d)$ be a metric space. A Borel measure $\mu$ 
	on $X$ is   {\em $\tau$-additive} if for every
	collection  $\{O_\alpha : \alpha \in \Lambda\}$
	of  open sets, we have
	$$
	\mu (\cup_\alpha O_\alpha) = \sup_{\mathcal{F}} \mu(\cup_{i=1}^n O_{\alpha_i}),
	$$
	where the supremum is taken over all finite subcollections $\mathcal{F} = \{O_{\alpha_1}, \dots, O_{\alpha_n} \}$
	of  $\{O_\alpha : \alpha \in \Lambda\}$.
	If $\mu$  assigns finite measure
	to bounded Borel sets, we say it is {\em locally finite}.
	Finally, we call $(X, d, \mu)$  a {\em metric measure space} if
	$\mu$ is a  $\tau$-additive, locally finite  Borel measure on the metric space $(X, d)$. 
\end{definition}

Note that the class of  $\tau$-additive measures  includes all  Borel measures on 
separable metric spaces and all Radon measures on arbitrary metric spaces. We will always assume that measures are not identically zero.

\vskip .2 cm

Seemingly natural statements which do hold in $\R^d$ regarding balls, centers and radii, can in general fail to be true. For instance, it may happen that $B(x,r) \subsetneqq
B(y, s)$, but $B(x,2 r) \supsetneqq
B(y, 2 s)$.

\begin{example}  
Fix $0 < \varepsilon \ll 1$. Consider first the subset of the real line $\{ - 2 - \varepsilon, 0, 1 + \varepsilon, 2, 2 + \varepsilon\}$, with the inherited distance. Then both $\{0, 1 + \varepsilon, 2\} = B^{\operatorname{cl}} (0, 2)$ and $\{0, 1 + \varepsilon, 2, 2 + \varepsilon\} = B^{\operatorname{cl}} ( 1 + \varepsilon, 1 + \varepsilon)$ are balls; note that   $B^{\operatorname{cl}} (0, 2)\subsetneqq  B^{\operatorname{cl}} ( 1 + \varepsilon, 1 + \varepsilon)$ and $B^{\operatorname{cl}} (0, 4)\supsetneqq  B^{\operatorname{cl}} ( 1 + \varepsilon, 2 + 2 \varepsilon)$.
\end{example}

\begin{definition} A Borel measure $\mu$ on $(X,d)$ is {\em doubling}  if there exists a 
	$C> 0 $ such that for all $r>0 $ and all $x\in X$, $\mu (B(x, 2 r)) \le C\mu(B(x,r))$. 
\end{definition}

Thus, the doubling condition applies to all possible choices of centers and radii of a ball $B$. 
Next we consider the weak doubling and tripling properties of  measures, and their natural generalization to other expansion factors.

\begin{definition} A Borel measure $\mu$ on $(X,d)$ has {\em weakly bounded $t$-growth} if there exist a $C \ge 1 $ and a $t > 1$ such that for every ball $B$  of $X$, it is possible to find an $x \in B$ and an $r > 0$ with
$B  = B(x,r)$ and  $\mu (B(x, t r)) \le C\mu(B(x,r))$. If $t = 2$ we say that $\mu$  is {\em weakly doubling}, while if   $t = 3$ we call $\mu$ {\em weakly tripling}. By analogy with doubling, we will use {\em weakly $t$-bling} as an abbreviation to refer to measures with  weakly bounded $t$-growth. 
\end{definition}

Recall that if a measure is doubling, no ball can have measure zero. On the other hand, we have the following

\begin{theorem} \label{counter007} There exists a separable  ultrametric space for which all measures are weakly $t$-bling with constant $C = 1$, for every $t >1$.
\end{theorem}

\begin{proof} Let $X :=  \mathbb{N}$ with the discrete metric
	$d(i,i):= 0$, and for $i \ne j$, 
	$d(i,j):= 1$. Since the only balls are the singletons and the whole space, it is enough to set $X := B^{\operatorname{o}}(0, 2)$ and  $\{n\} := B^{\operatorname{o}}(n, 1/t)$ for each $n\in X$. It follows that for every ball $B \subset X$ with the indicated radius, we have $B = tB$.
\end{proof}

While a space supporting a doubling measure is automatically separable, this is not the case if $\mu$ is only weakly $t$-bling: just consider the preceding proof but taking $X =  \mathbb{R}$.

\begin{example} A  measure can be weakly $t$-bling for every $t \in (1,3)$ and still fail to be weakly tripling. Let $X  \subset \R$ (with the inherited distance) be defined as follows: set $x_{-3} := - 1$, $x_{-2} := - 1/3 $, $x_{-1} := 0$, $x_{0} := 4^0$, and for $n\ge 1$, 
$x_{n} := \sum_{k = 0}^n 4^k$, so $x_{n+1} - x_{n} = 4^{n + 1}$. Let $X := \{x_{n} : n \ge -3\}$ be the metric space under consideration. 
We define a discrete measure $\mu$ on $X$, by setting $\mu \{x_{-3}\} = \mu \{x_{-2}\} = \mu \{x_{-1}\} := 1$ and for $n\ge 0$, $\mu \{x_{n}\} := 2^{2^n}$. To see that $\mu$ is not weakly tripling, take $n\gg 1$
and consider the ball 
 $B = B^{\operatorname{cl}} (x_{n},  4^{n + 1}/3) = \{x_{-2},  \dots, x_{n}\}$. Note that its only center is $x_{n}$, so its smallest possible radius is 
 $x_{n} - x_{-2} = 4^{n + 1}/3$. Tripling the radius yields $x_{ n + 1} \in  B^{\operatorname{cl}} (x_{n}, 4^{n + 1})$, and since $\mu(B^{\operatorname{cl}} (x_{n},  4^{n + 1}/3) )/ \mu( \{x_{n + 1}\}) \to 0$ as $n\to \infty$, the weakly tripling condition fails. 
 
 Next, fix $t \in (1,3)$; we may, without loss of generality, assume that $t > 2$.
Let $B =  \{x_{j},  \dots, x_{n}\}$ be an arbitrary ball, with the indices in increasing order and $j \le n$ (so $B$ could be just one point). Note that if $j = n$,  $B =  \{x_{j}\}$ can always be expressed as $B^{\operatorname{cl}} (x_{j}, 1/10)$, while if  $j < n$ and  $n \ge 0$,   then we can write $B = B^{\operatorname{cl}}( x_{n},  x_{n} -  x_{j})$, so these sets are indeed balls (for $n < 0$ the intervals $ \{x_{j},  \dots, x_{n}\}$ are also balls, as can be checked directly).
If  $j = -3$, so 
 $x_{j} = -1$, we write $B = B^{\operatorname{cl}}( x_{j},  x_{n} -  x_{j})$. Then $B = t B$ provided $n >1$, while if $n\le 1$ we suitably adjust the weakly $t$-bling constant $C$, as there are only finitely many possibilites. Finally,
 if   $j > -3$ then
 $x_{n}$ is the unique center of $B$ and we write $B = B^{\operatorname{cl}}( x_{n},  x_{n} -  x_{j})$. Since $t B$ does not contain points strictly to the right of $x_n$, we conclude that $\mu$ is weakly $t$-bling.
\end{example}

The standard covering argument shows that weak tripling with constant $C$ is enough to ensure the
weak type $(1,1)$ of the uncentered maximal operator, with the same constant. An additional iteration allows us
to relax the hypothesis to weak $(1 + \sqrt 2)$-bling, but the constant we obtain is $C^2$. We recall the relevant definitions regarding maximal operators.

\begin{definition}\label{maxfun} Let $(X, d, \mu)$ be a metric measure space and let $g$ be  a locally integrable function 
	on $X$. The {\em  centered Hardy-Littlewood maximal operator} $M_{\mu}$ is given by 
	\begin{equation*}
	M_{ \mu} g(x) := \sup _{\{r > 0 \ : \ 0 < \mu (B(x, r)) \}} \frac{1}{\mu
		(B(x, r))} \int _{B(x, r)} \vert g\vert d\mu,
	\end{equation*}
	while the {\em  uncentered Hardy-Littlewood maximal operator} $M_{\mu}^u$ is defined as 
	
	\begin{equation*}
	M_{ \mu}^u g(x) := \sup _{\{B : \ 0 < \mu (B) \mbox{ and } x \in B \}} \frac{1}{\mu
		(B)} \int _{B} \vert g\vert d\mu.
	\end{equation*}
	We say that  $M_{\mu}$ satisfies a
		weak type $(1,1)$ inequality if there exists a constant $c > 0$ such that
		\begin{equation*}
		\mu (\{M_{\mu} g > \alpha\}) \le \frac{c \|g\|_{L^1(\mu)}}{\alpha},
		\end{equation*}
		where $c=c(\mu)$ depends neither on $g\in L^1 (\mu)$
		nor on $\alpha > 0$. The lowest constant $c$ that satisfies the preceding
		inequality is denoted by $\|M_{\mu}\|_{L^1\to L^{1, \infty}}$. The definitions and notations regarding the weak type (1,1) are entirely analogous for the uncentered maximal operator, and more generally, for arbitrary sublinear operators.
	\end{definition}

\begin{theorem}\label{HLMF}  Let $(X, d, \mu)$ be a metric measure
 space, and let $\mu$ be  weakly $(1 + \sqrt 2)$-bling  with constant $C$. Then the uncentered Hardy-Littlewood maximal operator $M^u_{\mu}$ associated with $\mu$ satisfies $\|M^u_{\mu}\|_{L^1  \to L^{1,\infty}}  \le  C^2$.
\end{theorem}

\begin{proof} Fix $\varepsilon > 0$, let $a > 0$, and let $f\in L^1(\mu)$ be such that $\|f\|_1 > 0$. 
A priori we might have $\mu \{M^u_{\mu} f > a\} < \infty$ or  $\mu \{M^u_{\mu} f > a\} = \infty$. After dealing with the case $\mu \{M^u_{\mu} f > a\} < \infty$, we will show that   $\mu \{M^u_{\mu} f > a\} = \infty$ cannot 
happen. Of course, if $\mu \{M^u_{\mu} f > a\} = 0$ there is nothing to prove, so we can assume that $a >0$ is such that $\mu \{M^u_{\mu} f > a\} > 0$. Now, if $\mu \{M^u_{\mu} f > a\} < \infty$, 
 it is enough to find  a disjoint collection of balls $\{B^{\operatorname{o}}_{i_1}, \dots,   B^{\operatorname{o}}_{i_m}\}$ with  $a \mu B^{\operatorname{o}}_{i_j} < \int_{B^{\operatorname{o}}_{i_j}} |f|$ and 	$$
	(1 - \varepsilon) \mu \{M^u_{\mu} f > a\} 
	< 
	C^2 \mu \cup_{j= 1}^m B^{\operatorname{o}}_{i_j}.
	$$
Once we have this disjoint collection, the rest of the proof is standard.

For each $x\in \{M^u_{\mu} f > a\}$ select a $B^{\operatorname{o}}_x$ containing $x$, such that $a \mu B^{\operatorname{o}}_x < \int_{B^{\operatorname{o}}_x} |f|$.
	By the $\tau$-additivity of $\mu$, there is a finite subcollection 
	$\{B^{\operatorname{o}}_i :  1 \le i \le N\}$ 
	of balls with positive measure  satisfying	$$
	(1 - \varepsilon) \mu \{M^u_{\mu} f > a\} 
	=
	(1 - \varepsilon) \mu \cup \{ B^{\operatorname{o}}_x : x\in  \{M^u_{\mu} f > a\}\}
	< 
	\mu \cup_{i= 1}^N B^{\operatorname{o}}_i.
	$$
Next, for each $i = 1, \dots , N$ select a center $z_i$ and a radius $r_i$ such that $B^{\operatorname{o}}_i = B^{\operatorname{o}} (z_i, r_i)$ and $\mu B^{\operatorname{o}} (z_i, (1 + \sqrt 2) r_i) \le C \mu B^{\operatorname{o}} (z_i, r_i)$. Reorder the balls by non-increasing radii $r_i$,  and apply the standard selection procedure to obtain a disjoint subcollection:  let $B^{\operatorname{o}}_{i_1} $ be the first ball in the reordered list, select $B^{\operatorname{o}}_{i_2}$
to be the first ball in the list not intersecting $B^{\operatorname{o}}_{i_1}$, and in general, choose 
$B^{\operatorname{o}}_{i_k}$ to be the first ball in the list not intersecting any of the previously selected balls.
Then the process finishes after a finite number of steps with, let's say, $B^{\operatorname{o}}_{i_m}$.

We need to control the mass lost with the balls not chosen. Recall that  $\mu B^{\operatorname{o}} (z_{i_1}, (1 + \sqrt 2) r_{i_1}) \le C \mu B^{\operatorname{o}} (z_{i_1}, r_{i_1})$. By another application of weak $(1 + \sqrt 2)$-bling, this time to $B^{\operatorname{o}} (z_{i_1}, (1 + \sqrt 2) r_{i_1})$,  we conclude that there are a $z \in B^{\operatorname{o}} (z_{i_1}, (1 + \sqrt 2) r_{i_1})$ and an $r > 0$ such that $B^{\operatorname{o}} (z_{i_1}, (1 + \sqrt 2) r_{i_1}) = B^{\operatorname{o}} (z,  r)$ and 
$ \mu B^{\operatorname{o}} (z, (1 + \sqrt 2) r) \le C \mu B^{\operatorname{o}} (z, r)$, whence $ \mu B^{\operatorname{o}} (z, (1 + \sqrt 2) r) \le C^2 \mu B^{\operatorname{o}}_{i_1}$.

Let 
$B^{\operatorname{o}}(z_{1,1}, r_{1,1}), \dots,   B^{\operatorname{o}}(z_{1, k}, r_{1, k})$ be the collection of all balls from $\{B^{\operatorname{o}}_i = B^{\operatorname{o}} (z_i, r_i) :  1 \le i \le N\}$  that intersect $B^{\operatorname{o}}_{i_1}$ and are different from it. By the selection procedure we have $r_{i_1} \ge r_{1,1}, \dots,  r_{1,k}$, so $z_{1,1}, \dots,  z_{1,k} \in B^{\operatorname{o}} (z_{i_1}, (1 + \sqrt 2) r_{i_1})$. We claim that  $\cup_{s=1}^k  B^{\operatorname{o}}(z_{1,s}, r_{1,s}) \subset B^{\operatorname{o}} (z, (1 + \sqrt 2) r)$. If 
 $\cup_{s=1}^k  B^{\operatorname{o}}(z_{1,s}, r_{1,s}) \subset B^{\operatorname{o}} (z,  r) = B^{\operatorname{o}} (z_{i_1}, (1 + \sqrt 2) r_{i_1})$ there is nothing to show, so assume that for some $s$ we can find a $w \in B^{\operatorname{o}}(z_{1,s}, r_{1,s}) \setminus B^{\operatorname{o}} (z_{i_1}, (1 + \sqrt 2) r_{i_1})$.  Then $z_{1,s} \notin B^{\operatorname{o}} (z_{i_1}, \sqrt 2 r_{i_1})$, for otherwise 
we would have $w \in B^{\operatorname{o}} (z_{i_1}, (1 + \sqrt 2) r_{i_1})$. Hence $d(z_{1,s}, z_{i_1}) \ge \sqrt 2 r_{i_1}$ and thus $r > r_{i_1} /\sqrt 2$ by the triangle inequality, for both $z_{1,s}, z_{i_1} \in  
B^{\operatorname{o}} (z,  r)$. But now $d(w, z) \le d(w, z_{1,s}) + d(z_{1,s}, z) < r_{1,s} + r \le (\sqrt 2 + 1) r.$  
By repeating the process with each of the selected balls $B^{\operatorname{o}}_{i_2}, \dots, B^{\operatorname{o}}_{i_m}$, we conclude that $\mu \cup_{i= 1}^N B^{\operatorname{o}}_i \le C^2 \mu \cup_{j= 1}^m B^{\operatorname{o}}_{i_j}.$

Finally, using the same notation as above, 
if $\mu \{M^u_{\mu} f > a\} = \infty$, then it is possible to find,  
by the $\tau$-additivity of $\mu$,  a finite subcollection 
	$\{B^{\operatorname{o}}_i :  1 \le i \le N\}$ 
	of balls with positive measure  satisfying	$
\frac{2  C^2 \|f\|_1}{a}
	< 
	\mu \cup_{i= 1}^N B^{\operatorname{o}}_i.
	$
Once we have this finite collection, we apply the disjointification process used before to obtain 
$$ \frac{2  C^2 \|f\|_1}{a}
<
\mu \cup_{i= 1}^N B^{\operatorname{o}}_i
\le
C^2 \mu \cup_{j= 1}^m B^{\operatorname{o}}_{i_j} 
$$
$$
= C^2 \sum_{j= 1}^m \mu  B^{\operatorname{o}}_{i_j} <  
C^2 \sum_{j= 1}^m \frac{1 }{a} \int_{ B^{\operatorname{o}}_{i_j} } |f| d \mu
\le \frac{C^2 \|f\|_1}{a}.$$
\end{proof}

I do not know if the expansion factor $(1 + \sqrt 2)$ can be lowered in the preceding result.

\vskip .2 cm

An anonimous referee asks for an explanation about the source of  $(1 + \sqrt 2)$. To see where it comes from, consider the example of two intersecting balls $B(x,1)$ and $B(y, 1)$. Let $t > 1$, and suppose that there is a $w \in B(y, 1) \setminus B(x , 1 + t)$, so $d(x,y) \ge t$. In order to apply the weak $t$-bling condition, we can choose suitable $z \in B(x, 1 + t)$ and $r > 0$ such that $B(z,r) =  B(x , 1 + t)$ and $t$-bling holds for that choice. We want to ensure that $w \in B(z,(1 + t) r)$. By the triangle inequality, $d(z, w) \le d(z, y) + d(y,w) \le r + 1$. Hence, we select $t$ as small as possible but satisfying $r + 1 \le (1 + t)r$, or equivalently, $1 \le t r$. Since all the information we have about $r$ is that $r\ge t/2$, we select the smallest $t$ such that $1 \le t (t/2)$, and prove that this choice works.

\vskip .2 cm

Next we motivate a question formulated by a second  anonimous referee. One way of obtaining weak type bounds without doubling or Besicovitch conditions is to make the maximal operators smaller, via the  division by the measure  of a ball with radius larger than the radius of the ball one integrates over. This was done for the centered maximal operator in \cite{NTV},  as part of the project to  develop a Calder\'on-Zygmund theory without doubling. 
The centered operator  defined in \cite{NTV} used the expression $ \frac{1}{\mu
	(B(x, 3 r))} \int _{B(x, r)} \vert g\vert d\mu$, instead of the standard quotient
  $ \frac{1}{\mu
	(B(x, r))} \int _{B(x, r)} \vert g\vert d\mu$, and the resulting maximal operator was easily shown to be of weak type (1,1) with optimal constant 1.

The properties of analogous operators with more general expansion factors in place of 3, have been further explored in \cite{Sa}, \cite{Te}, \cite{Ste}, \cite{Ste1}. More specifically, the smallest expansion factors that still yield the weak type bounds are different in the centered and uncentered case, cf. \cite{Sa}, \cite{Te}, \cite{Ste}.

 In this context the second referee asks whether one can find a result related to Theorem \ref{HLMF}, but dealing with the centered maximal operator and using an expansion factor  smaller than $1 + \sqrt 2$. We have not been able to obtain weak type boundedness under a weaker $t$-bling assumption. An obstacle to the use of a smaller $t$ comes from the fact that by doubling the radius one might be able to cover one center of a ball, but perhaps not all its centers.
In other words, the possible non-uniqueness of centers can blur the lines separating the centered from the uncentered case. For instance, in ultrametric sapces every point in a ball is a center, so the centered and the uncentered maximal operators are exactly the same.

We have seen that the assumption of weak tripling can be relaxed when dealing with the uncentered maximal operator.
As a partial answer to the second referee, we show  that  it is not possible to go below weak doubling in the centered case (and we do not know if weak doubling suffices).

\begin{example}  We modify \cite[Example 4.1]{Al1} to show that the weak type (1,1)  can fail under weak $t$-bling if $t < 2$. Furthermore we do not even need to consider the maximal operator, we can just use the averaging operator associated to balls of radius 1, given by
\begin{equation*}
A_{1} g(x) := \frac{1}{\mu
	(B^{\operatorname{cl}}(x, 1))} \int _{B^{\operatorname{cl}}(x, 1)}  g d\mu.
\end{equation*}

We define next a   path connected subset   $P\subset \mathbb{R}^2$, on which we use  the path metric  $d$
instead of the ambient space metric: the distance between
two points is the length of the shortest path joining them. 
Start with the positive $x$-axis. For each $n\ge 1$, select $2^n$ points from the circumference of radius 1 centered at
$(9^n, 0)$: $z_{n,1}\dots, z_{n,2^n}\in \{(x,y) \in\mathbb{R}^2 : (x - 9^n)^2 + y^2 = 1\}$. Join the points 
 $z_{n,1}\dots, z_{n,2^n}$ to the center $(9^n, 0)$ using straight line segments (radii), let $P$ be the union of 
 the positive $x$-axis with all the ``spikes" attached to the centers $(9^n, 0)$, and  let $\mu$ be the counting
measure on the points $(9^n, 0)$ and  $z_{n,1}\dots, z_{n,2^n}$, for every $n\ge 1$. Now $d(z_{n,k}, (9^n, 0)) = 1$, while
if $m\ne n$,  $d(z_{n,k}, z_{m,j}) \ge 9$. Thus,  $\mu B^{\operatorname{cl}}(z_{n,k}, 1) = 2$ and $\mu B^{\operatorname{cl}}((9^n, 0), 1) = 2^n + 1$. Next, we define a metric space $X \subset P$ with the inherited distance as follows: $X := \{x \in P : \mu (\{x\}) = 1\}$.
Then  $\|\mathbf{1}_{\{(9^n, 0) \}}\|_{L^1} = 1$ and $(A_{1} \mathbf{1}_{\{(9^n, 0) \}}) (z_{n, k}) = 1/2$, from whence it follows that
 $\|A_{1}\|_{L^1\to L^{1, \infty}} \ge  2^{n - 1}$ for every $n \ge 1$.
 
 Finally, we check that whenever $ t \in (1,2)$, the measure $\mu$ is weakly $t$-bling (to see that is not weakly doubling, just consider the balls $B^{\operatorname{cl}}(z_{n,k}, 1)$, which contain exactly 2 points, have $z_{n,k}$ as its unique center, and 1 as its smallest possible radius). For singletons $\{x\}$ we just write $\{x\} = B^{\operatorname{cl}}(x, 1/3)$ and note that $\{x\} = B^{\operatorname{cl}}(x, (1 + t)/3)$. We also have $B^{\operatorname{cl}}( (9^n, 0)), 1) = B^{\operatorname{cl}}( (9^n, 0)), 1 + t)$ and
 $B^{\operatorname{cl}}(z_{n,k}, 1) = B^{\operatorname{cl}}(z_{n,k}, 1 + t)$. The case left to consider is that of balls with large radii. 
 For (closed) balls $B$ intersecting more than one ball $B^{\operatorname{cl}}( (9^n, 0)), 1)$, say, $B \cap  B^{\operatorname{cl}}( (9^i, 0)), 1) \ne \emptyset$ and $B \cap B^{\operatorname{cl}}( (9^j, 0)), 1) \ne \emptyset$, where
 $i < j$ and these are respectively the smallest and largest indices for which the intersection is nonempty, 
 we select any center $x$ and the smallest possible radius $r$ for which  $B = B^{\operatorname{cl}}(x,r)$. Then $r \le d( B^{\operatorname{cl}}( (9^1, 0)), 1),  B^{\operatorname{cl}}( (9^j, 0)), 1) < 9^j$, so if $B^{\operatorname{cl}}(x, 2r) \ne B(x,r)$, then the new points in the difference set  must all come from balls  $B^{\operatorname{cl}}( (9^m, 0)), 1)$ with $m \le i$. Since $\mu  B^{\operatorname{cl}}( (9^j, 0)), 1)$ is comparable to $\sum_1^j \mu  B^{\operatorname{cl}}( (9^k, 0)), 1)$, we conclude that $\mu$ is weakly $t$-bling.
 \end{example}

\begin{example}  We explain with some detail why the usual arguments that yield weak type bounds under weaker hypotheses for the centered case, fail in our context. Take $X := (-2, -1] \cup \{0\} \cup [1,2)$ with the metric inherited from the real line, and with the measure $\mu$ given by the sum of the restriction of Lebesgue measure plus $\delta_0$. By the standard covering argument for the real line, the uncentered operator is of weak type (1,1) with bound at most 2, for every Borel measure. Now 
$\{M_{\mu} \mathbf{1}_{\{0\}} \ge 1/2\} = X$. For each $x\in [1,2)$ we see that $M_{\mu} \mathbf{1}_{\{0\}} (x)
= \frac{1}{\mu B^{\operatorname{cl}}(x,x) } \int_0^{2x} \mathbf{1}_{\{0\}} (t) dt$, for each $x\in (-2, -1]$ we make the symmetric choice of averaging ball, and finally, for $0$ we select $B^{\operatorname{cl}}(0,1/10) $. Then we get a covering collection with $3$ different balls but uncountably many names. Furthermore, if we apply the operations of doubling the radius, the ball $ \{0\} \cup [1,2)$ will give rise to uncountably many different balls $B^{\operatorname{cl}}(x, 2 x)$. Likewise,
if one wants to cover this set with balls of small radii, as in the Stein-Str\"{o}mberg argument (cf. \cite[Lemma]{StSt}), then we also get uncountably many different balls $B^{\operatorname{cl}}(x, t x)$, where $0 < t \ll 1$. While $\mu$ is doubling, the preceding example can be modified, changing both the space and the measure, so for balls such as $ \{0\} \cup [1,2)$, one may be forced to select a small radius, as per the weak doubling hypothesis. But then small, far away balls  $B^{\operatorname{cl}}(x, t x)$ may fail to be covered when the radius is doubled.
\end{example}


\begin{thebibliography}{WWW}

\bibitem[Al]{Al}  Aldaz, J. M.  {\em Kissing numbers and the centered maximal operator.} J. Geom. Anal. 31 (2021), no. 10, 10194--10214. Available at the Mathematics ArXiv.


\bibitem[Al1]{Al1}  Aldaz, J. M.  {\em Local comparability of measures, averaging and  maximal averaging operators.}  Potential Anal. 49 (2018), no. 2, 309--330. Available at the Mathematics ArXiv. 
 


\bibitem[He]{He}   Heinonen, Juha {\em Lectures on analysis on metric spaces.} Universitext. Springer-Verlag, New York, 2001.

 
\bibitem[KlTo]{KlTo} Kleinbock, Dmitry; Tomanov, George 
{\em Flows on S-arithmetic homogeneous spaces and applications to metric Diophantine approximation.} Comment. Math. Helv. 82 (2007), no. 3, 519--581. 

\bibitem[NTV]{NTV} Nazarov, F.; Treil, S.; Volberg, A. 
{\em Weak type estimates and Cotlar inequalities for Calderón-Zygmund operators on nonhomogeneous spaces.}
Internat. Math. Res. Notices 1998, no. 9, 463--487.


 
 \bibitem[LeRi]{LeRi} Le Donne, Enrico; Rigot, S\'everine {\em Besicovitch covering property for homogeneous distances on the Heisenberg groups.}  J. Eur. Math. Soc. (JEMS) 19 (2017), no. 5, 1589--1617.
 
  \bibitem[LeRi2]{LeRi2} Le Donne, Enrico; Rigot, S\'everine {\em Besicovitch Covering Property on graded groups and applications
  	to measure differentiation,}  J. Reine Angew. Math. 750 (2019), 241--297.
  
   \bibitem[McC]{McC}  J. McCleary {\em Geometry from a differentiable viewpoint,}
  Cambridge University Press, Cambridge, 1994.
  


 \bibitem[Ri]{Ri}  Rigot, S\'everine  
  {\em Differentiation of measures in metric spaces.} To appear in a volume of the Lecture Notes in Mathematics, C.I.M.E. Foundation Subseries, arXiv:1802.02069.
 
\bibitem[Sa]{Sa}  Sawano, Yoshihiro {\em Sharp estimates of the modified Hardy-Littlewood maximal operator on the nonhomogeneous space via covering lemmas.} Hokkaido Math. J. 34 (2005), no. 2, 435–-458. 

  \bibitem[StSt]{StSt} Stein, E. M.; Str\"{o}mberg, J. O. {\em Behavior of maximal functions in $R\sp{n}$
 for large $n$.} Ark. Mat. 21 (1983), no. 2, 259--269.

\bibitem[Ste]{Ste}  Stempak, Krzysztof {\em Modified Hardy-Littlewood maximal operators on nondoubling metric measure spaces.} Ann. Acad. Sci. Fenn. Math. 40 (2015), no. 1, 443–448.

\bibitem[Ste1]{Ste1} Stempak, Krzysztof {\em Examples of metric measure spaces related to modified Hardy-Littlewood maximal operators.} Ann. Acad. Sci. Fenn. Math. 41 (2016), no. 1, 313--314.

\bibitem[Te]{Te} Terasawa, Yutaka {\em Outer measures and weak type (1,1) estimates of Hardy-Littlewood maximal operators.} J. Inequal. Appl. 2006, Art. ID 15063, 13 pp.



\end{thebibliography}
\end{document}